\documentclass[11pt]{amsart}

\newtheorem{theorem}{Theorem}[section]

\newtheorem{prop}[theorem]{Proposition}

\theoremstyle{definition}

\theoremstyle{remark}

\numberwithin{equation}{section}

\begin{document}

\newcommand{\spacing}[1]{\renewcommand{\baselinestretch}{#1}\large\normalsize}
\spacing{1.14}

\title{On the Randers metrics on two-step homogeneous nilmanifolds of dimension five}

\author {H. R. Salimi Moghaddam}

\address{Department of Mathematics, Faculty of  Sciences, University of Isfahan, Isfahan,81746-73441-Iran.} \email{salimi.moghaddam@gmail.com and hr.salimi@sci.ui.ac.ir}
\keywords{invariant metric, flag curvature,
Berwald space, Randers  space, two-step nilpotent Lie group\\
AMS 2000 Mathematics Subject Classification: 22E60, 53C60, 53C30.}


\begin{abstract}
In this paper we study the geometry of simply connected two-step
nilpotent Lie groups of dimension five. We give the Levi-Civita
connection, curvature tensor, sectional and scalar curvatures of
these spaces and show that they have constant negative scalar
curvature. Also we show that the only space which admits left
invariant Randers metric of Berwald type has three dimensional
center. In this case the explicit formula for computing flag
curvature is obtained and it is shown that flag curvature and
sectional curvature have the same sign.
\end{abstract}

\maketitle


\section{\textbf{Introduction}}
A connected Riemannian manifold which admits a transitive
nilpotent Lie group N of isometries is called a nilmanifold. E.
Wilson showed that for a given homogeneous nilmanifold $M$, there
exists a unique nilpotent Lie subgroup $N$ of $I(M)$ acting simply
transitively on $M$, and $N$ is normal in $I(M)$ (see \cite{Wi}).
Therefore the Riemannian manifold $M$ will be identified with the
Lie group $N$ equipped with a left-invariant Riemannian metric $g$.\\
Between nilpotent Lie groups the family of two-step nilpotent Lie
groups endowed with left invariant Riemannian metrics are studied
specially in the recent years (see \cite{Eb1}, \cite{Eb2} and
\cite{HoKo}.). \\
J. Lauret classified all homogeneous nilmanifolds of dimension
three and four in \cite{La}. Also simply connected two-step
nilpotent Lie groups of dimension five equipped with left
invariant Riemannian metrics are classified by S. Homolya and O.
Kowalski in \cite{HoKo}. For this reason they classified metric
Lie algebras with one, two and three dimensional center. We use
their results in this article.\\
On the other hand studying invariant Finsler metrics on lie groups
and homogeneous spaces developed in these years, (for example see
\cite{DeHo1,DeHo2,DeHo3,EsSa1,EsSa2,Sa1,Sa2,Sa3,Sa4,Sa5,Sa6,ToKo}.).\\
A. T\'{o}th and Z. Kov\'{a}cs studied the geometry of some special
two-Step nilpotent groups with left invariant Finsler metrics in
\cite{ToKo}. In the present paper we study the geometry of simply
connected two-step nilpotent Lie groups of dimension five endowed
with left invariant Riemannian metrics. Then we discuss the
existence of left invariant Randers metrics of Berwald type (an
interesting Finsler metric which have many applications in
physics) on these spaces. Finally we give the explicit formula for
computing flag curvature of these metrics and show that the flag
curvature and sectional curvature of the base left invariant
Riemannian metric have the same sign.


\section{\textbf{Preliminaries}}
In this section we give some preliminaries about invariant
(Riemannian and Finsler) metrics.\\
A Riemannian metric $g$ on a Lie group $G$ is called left
invariant if
\begin{eqnarray}
  g(a)(Y,Z)=g(e)(T_al_{a^{-1}}Y,T_al_{a^{-1}}Z), \ \ \ \ \ \forall
  a\in G, \forall Y,Z\in T_aG,
\end{eqnarray}
where $e$ is the unit element of $G$.\\

For a Lie group $G$ equipped with a left invariant Riemannian
metric $g$ the Levi-Civita connection is defined by the following
formula:
\begin{eqnarray}
 2<\nabla_UV,W>=<[U,V],W>-<[V,W],U>+<[W,U],V>,\label{Levi-Civita}
\end{eqnarray}
for any $U,V,W\in\frak{g}$, where $\frak{g}$ is the Lie algebra of
$G$ and $< ,>$ is the inner product induced by $g$ on $\frak{g}$.\\
Now we can generalize this definition to Finsler manifolds.\\
A Finsler metric on a manifold $M$ is a non-negative function
$F:TM\longrightarrow\Bbb{R}$ with the following properties:
\begin{enumerate}
    \item $F$ is smooth on the slit tangent bundle
    $TM^0:=TM\setminus\{0\}$,
    \item $F(x,\lambda Y)=\lambda F(x,Y)$ for any $x\in M$,
    $Y\in T_xM$ and $\lambda >0$,
    \item the $n\times n$ Hessian matrix $[g_{ij}]=[\frac{1}{2}\frac{\partial^2 F^2}{\partial y^i\partial
    y^j}]$ is positive definite at every point $(x,Y)\in TM^0$.
\end{enumerate}

A special type of Finsler metrics are Randers metrics which have
been introduced by G. Randers \cite{Ra} in his research on general
relativity. Randers metrics are constructed on Riemannian metrics
and vector fields (1-forms).\\
Let $g$ and $X$ be a Riemannian metric and a vector field on a
manifold $M$ respectively such that $\|X\|=\sqrt{g(X,X)}<1$. Then
a Randers metric $F$, defined by $g$ and $X$, is a Finsler metric
as follows:
\begin{eqnarray}\label{Randers}
     F(x,Y)=\sqrt{g(x)(Y,Y)}+g(x)(X(x),Y), \ \ \ \ \forall x\in M,
     Y\in T_xM.\label{IRM}
\end{eqnarray}
Similar to the Riemannian case, a Finsler metric $F$ on a Lie
group $G$ is called left invariant if
\begin{eqnarray}
  F(a,Y)=F(e,T_al_{a^{-1}}Y)  \ \ \ \ \ \ \ \forall a \in G, Y\in T_aG.
\end{eqnarray}

A special family of Randers metrics (or in general case Finsler
metrics) is the family of Berwaldian Randers metrics. A Randers
metric of the form \ref{IRM} is of Berwald type if and only if the
vector field $X$ is parallel with respect to the Levi-Civita
connection of $g$. In these metrics the Chern connection of the
Randers metric $F$ coincide with the Levi-Civita connection of the Riemannian metric $g$.\\
One of the important quantities which associates with a Finsler
manifold is flag curvature which is a generalization of sectional
curvature to Finsler manifolds. Flag curvature is defined as
follows.

\begin{eqnarray}\label{flag}
  K(P,Y)=\frac{g_Y(R(U,Y)Y,U)}{g_Y(Y,Y).g_Y(U,U)-g_Y^2(Y,U)},
\end{eqnarray}
where $g_Y(U,V)=\frac{1}{2}\frac{\partial^2}{\partial s\partial
t}(F^2(Y+sU+tV))|_{s=t=0}$, $P=span\{U,Y\}$,
$R(U,Y)Y=\nabla_U\nabla_YY-\nabla_Y\nabla_UY-\nabla_{[U,Y]}Y$ and
$\nabla$ is the Chern connection induced by $F$ (see \cite{BaChSh}
and \cite{Sh1}.).\\

From now we consider $N$ is a simply connected two-step nilpotent
Lie group of dimension five and $\frak{n}$ is its Lie algebra.\\

\section{\textbf{Lie algebras with 1-dimensional center}}


We can use left invariant Riemannian metrics and left invariant
vector fields for constructing left invariant Randers metrics on
Lie groups.\\
Suppose that $G$ is Lie group, $g$ is a left invariant Riemannian
metric and $X$ is a left invariant vector field  on $G$ such that
$\sqrt{g(X,X)}<1$. Then we can define a left invariant Riemannian
metric on $G$ by using formula \ref{IRM}. In fact we have the
following proposition on these metrics.
\begin{prop}\label{Randers Lie group}
There is a one-to-one correspondence between the invariant Randers
metrics on the Lie group $G$ with the underlying Riemannian metric
$g$ and the left invariant vector fields with length $<1$.
Therefore the invariant Randers metrics are one-to-one
corresponding to the set
\begin{eqnarray}
  V=\{X\in\frak{g}|<X,X><1\}.
\end{eqnarray}
\end{prop}
\begin{proof}
It is suffix to let $H=\{e\}$ in theorem 2.2 of \cite{DeHo2}.
\end{proof}

In this section we consider the Lie algebra $\frak{n}$ has
$1-$dimensional center.\\
In \cite{HoKo} S. Homolya and O. Kowalski  showed that there exist
an orthonormal basis $\{e_1,e_2,e_3,e_4,e_5\}$ of $\frak{n}$ such
that
\begin{eqnarray}
  [e_1,e_2]=\lambda e_5 \ \ , \ \ [e_3,e_4]=\mu e_5,
\end{eqnarray}
where $\{e_5\}$ is a basis for the center of $\frak{n}$, and
$\lambda\geq\mu>0$. Also it is considered that the other
commutators are zero.\\

By using the equation \ref{Levi-Civita} for the Levi-Civita
connection we have:

\begin{eqnarray}\label{Levi-Civita case1}
 \begin{tabular}{|c|c|c|c|c|c|}
  \hline
    & $e_1$ & $e_2$ & $e_3$ & $e_4$ & $e_5$ \\
  \hline
  $\nabla_{e_1}$ & 0 & $\frac{\lambda}{2}e_5$ & 0 & 0 & $-\frac{\lambda}{2}e_2$ \\
  \hline
  $\nabla_{e_2}$ & $-\frac{\lambda}{2}e_5$ & 0 & 0 & 0 & $\frac{\lambda}{2}e_1$ \\
  \hline
  $\nabla_{e_3}$ & 0 & 0 & 0 & $\frac{\mu}{2}e_5$ & $-\frac{\mu}{2}e_4$ \\
  \hline
  $\nabla_{e_4}$ & 0 & 0 & $-\frac{\mu}{2}e_5$ & 0 & $\frac{\mu}{2}e_3$ \\
  \hline
  $\nabla_{e_5}$ & $-\frac{\lambda}{2}e_2$ & $\frac{\lambda}{2}e_1$ & $-\frac{\mu}{2}e_4$ & $\frac{\mu}{2}e_3$ & 0 \\
  \hline
\end{tabular}
\end{eqnarray}

The above equations for the Levi-Civita connection show that
curvature tensor is as follows.

\begin{eqnarray}\label{Curvature case1}
  \begin{tabular}{|c|c|c|c|c|c|}
    \hline
         & $e_1$ & $e_2$ & $e_3$ & $e_4$ & $e_5$ \\
    \hline
    $R(e_1,e_2)$ & $\frac{3\lambda^2}{4}e_2$ & $-\frac{3\lambda^2}{4}e_1$ & $\frac{\lambda\mu}{2}e_4$ & $-\frac{\lambda\mu}{2}e_3$ & 0 \\
    \hline
    $R(e_1,e_3)$ & 0 & $\frac{\lambda\mu}{4}e_4$ & 0 & $-\frac{\lambda\mu}{4}e_2$ & 0 \\
    \hline
    $R(e_1,e_4)$ & 0 & $-\frac{\lambda\mu}{4}e_3$ & $\frac{\lambda\mu}{4}e_2$ & 0 & 0 \\
    \hline
    $R(e_1,e_5)$ & $-\frac{\lambda^2}{4}e_5$ & 0 & 0 & 0 & $\frac{\lambda^2}{4}e_1$ \\
    \hline
    $R(e_2,e_3)$ & $-\frac{\lambda\mu}{4}e_4$ & 0 & 0 & $\frac{\lambda\mu}{4}e_1$ & 0 \\
    \hline
    $R(e_2,e_4)$ & $\frac{\lambda\mu}{4}e_3$ & 0 & $-\frac{\lambda\mu}{4}e_1$ & 0 & 0 \\
    \hline
    $R(e_2,e_5)$ & 0 & $-\frac{\lambda^2}{4}e_5$ & 0 & 0 & $\frac{\lambda^2}{4}e_2$ \\
    \hline
    $R(e_3,e_4)$ & $\frac{\lambda\mu}{2}e_2$ & $-\frac{\lambda\mu}{2}e_1$ & $\frac{3\mu^2}{4}e_4$ & $-\frac{3\mu^2}{4}e_3$ & 0 \\
    \hline
    $R(e_3,e_5)$ & 0 & 0 & $-\frac{\mu^2}{4}e_5$ & 0 & $\frac{\mu^2}{4}e_3$ \\
    \hline
    $R(e_4,e_5)$ & 0 & 0 & 0 & $-\frac{\mu^2}{4}e_5$ & $\frac{\mu^2}{4}e_4$ \\
    \hline
  \end{tabular}
\end{eqnarray}

Now we compute the sectional curvature and scalar curvature for
this Riemannian Lie group.\\
Let $\{A,B\}$ be an orthonormal basis for a two dimensional
subspace of $T_eN$, where $e$ is the unit element of $N$, as
follows:
\begin{eqnarray} \label{AB}
  A &=& ae_1+be_2+ce_3+de_4+fe_5 \\
  B &=&
  \tilde{a}e_1+\tilde{b}e_2+\tilde{c}e_3+\tilde{d}e_4+\tilde{f}e_5\nonumber
\end{eqnarray}
By using table \ref{Curvature case1} and above equations we have
\begin{eqnarray}
  R(B,A)A &=&(b\tilde{a}-a\tilde{b})(\frac{3\lambda^2}{4}(ae_2-be_1)+\frac{\lambda\mu}{2}(ce_4-de_3))\nonumber\\
  &&+(d\tilde{c}-c\tilde{d})(\frac{\lambda\mu}{2}(ae_2-be_1)+\frac{3\mu^2}{4}(ce_4-de_3))\nonumber\\
  &&+\frac{\lambda^2}{4}\{(f\tilde{b}-b\tilde{f})(fe_2-be_5)+(f\tilde{a}-a\tilde{f})(fe_1-ae_5)\} \nonumber\\
  &&+\frac{\mu^2}{4}\{(f\tilde{c}-c\tilde{f})(fe_3-ce_5)+(f\tilde{d}-d\tilde{f})(fe_4-de_5)\}\\
  &&+\frac{\lambda\mu}{4}\{(c\tilde{a}-a\tilde{c})(be_4-de_2)+(d\tilde{a}-a\tilde{d})(ce_2-be_3)\nonumber\\
  &&+(c\tilde{b}-b\tilde{c})(de_1-ae_4)+(d\tilde{b}-b\tilde{d})(ae_3-ce_1)\}\nonumber.
\end{eqnarray}

Now a direct computation shows that the sectional curvature $K^R$
can be obtained with the following equation.

\begin{eqnarray}
  K^R(A,B) &=& -\frac{3}{4}\{\lambda^2(a\tilde{b}-b\tilde{a})^2+\mu^2(c\tilde{d}-d\tilde{c})^2\} \nonumber\\
  &&+\frac{\lambda^2}{4}\{(f\tilde{a}-a\tilde{f})^2+(f\tilde{b}-b\tilde{f})^2\} \\
  &&+\frac{\mu^2}{4}\{(f\tilde{c}-c\tilde{f})^2+(f\tilde{d}-d\tilde{f})^2\}\nonumber\\
  &&+\frac{3\lambda\mu}{2}(a\tilde{b}-b\tilde{a})(d\tilde{c}-c\tilde{d})\nonumber
\end{eqnarray}
This formula shows that the Riemannian Lie group $N$ admits
positive and negative sectional curvature. This fact has been
proved by J. A. Wolf in \cite{Wol} which any non-commutative Lie
group admits positive and negative sectional curvatures.\\

Also for any $p\in N$ the scalar curvature is

\begin{eqnarray}\label{scalar case 1}
  S(p)=-\frac{\lambda^2+\mu^2}{2}<0,
\end{eqnarray}
which shows that this Riemannian manifold is of constant negative
scalar curvature.\\

\begin{prop}
There is not any left invariant Randers metric of Berwald type on
simply connected two-step nilpotent Lie groups of dimension five
with $1-$dimensional center.
\end{prop}
\begin{proof}
Proposition \ref{Randers Lie group} says that for an invariant
Randers metric we need a left invariant vector field with
length$<1$. On the other hand we know that a Randers metric is of
Berwald type if and only if the vector field is parallel with
respect to the Levi-Civita connection (see \cite{BaChSh}.). Let
$Q\in\frak{n}$ be a left invariant vector field on $N$ which is
parallel with respect to the Levi-Civita connection. By using
table \ref{Levi-Civita case1} and a direct computation we have
$Q=0$.
\end{proof}

\section{\textbf{Lie algebras with 2-dimensional center}}
In this section we study simply connected two-step nilpotent Lie
groups of dimension five equipped with left-invariant Riemannian
metric and $2-$dimensional center.\\
Let $N$ be as above. S. Homolya and O. Kowalski in \cite{HoKo}
showed that $\frak{n}$ admits an orthonormal basis
$\{e_1,e_2,e_3,e_4,e_5\}$ such that
\begin{eqnarray}
  [e_1,e_2]=\lambda e_4 \ \ , \ \ [e_1,e_3]=\mu e_5,
\end{eqnarray}
where $\{e_4,e_5\}$ is a basis for the center of $\frak{n}$, the
other commutators are zero and $\lambda\geq\mu>0$.\\

The equation \ref{Levi-Civita} and some computations shows that
the Levi-Civita connection of this Riemannian manifold is as the
following table.

\begin{eqnarray}\label{Levi-Civita case2}
 \begin{tabular}{|c|c|c|c|c|c|}
  \hline
      & $e_1$ & $e_2$ & $e_3$ & $e_4$ & $e_5$ \\
  \hline
  $\nabla_{e_1}$ & 0 & $\frac{\lambda}{2}e_4$ & $\frac{\mu}{2}e_5$ & $-\frac{\lambda}{2}e_2$ & $-\frac{\mu}{2}e_3$ \\
  \hline
  $\nabla_{e_2}$ & $-\frac{\lambda}{2}e_4$ & 0 & 0 & $\frac{\lambda}{2}e_1$ & 0 \\
  \hline
  $\nabla_{e_3}$ & $-\frac{\mu}{2}e_5$ & 0 & 0 & 0 & $\frac{\mu}{2}e_1$ \\
  \hline
  $\nabla_{e_4}$ & $-\frac{\lambda}{2}e_2$ & $\frac{\lambda}{2}e_1$ & 0 & 0 & 0 \\
  \hline
  $\nabla_{e_5}$ & $-\frac{\mu}{2}e_3$ & 0 & $\frac{\mu}{2}e_1$ & 0 & 0 \\
  \hline
\end{tabular}
\end{eqnarray}

Now by using the Levi-Civita connection we can compute the
curvature tensor as follows.

\begin{eqnarray}\label{Curvature case2}
  \begin{tabular}{|c|c|c|c|c|c|}
    \hline
     & $e_1$ & $e_2$ & $e_3$ & $e_4$ & $e_5$ \\
    \hline
    $R(e_1,e_2)$ & $\frac{3\lambda^2}{4}e_2$ & $-\frac{3\lambda^2}{4}e_1$ & 0 & 0 & 0 \\
    \hline
    $R(e_1,e_3)$ & $\frac{3\mu^2}{4}e_3$ & 0 & $-\frac{3\mu^2}{4}e_1$ & 0 & 0 \\
    \hline
    $R(e_1,e_4)$ & $-\frac{\lambda^2}{4}e_4$ & 0 & 0 & $\frac{\lambda^2}{4}e_1$ & 0 \\
    \hline
    $R(e_1,e_5)$ & $-\frac{\mu^2}{4}e_5$ & 0 & 0 & 0 & $\frac{\mu^2}{4}e_1$ \\
    \hline
    $R(e_2,e_3)$ & 0 & 0 & 0 & $\frac{\lambda\mu}{4}e_5$ & $-\frac{\lambda\mu}{4}e_4$  \\
    \hline
    $R(e_2,e_4)$ & 0 & $-\frac{\lambda^2}{4}e_4$  & 0 & $\frac{\lambda^2}{4}e_2$ & 0 \\
    \hline
    $R(e_2,e_5)$ & 0 & 0 & $-\frac{\lambda\mu}{4}e_4$ & $\frac{\lambda\mu}{4}e_3$ & 0 \\
    \hline
    $R(e_3,e_4)$ & 0 & $-\frac{\lambda\mu}{4}e_5$ & 0 & 0 & $\frac{\lambda\mu}{4}e_2$ \\
    \hline
    $R(e_3,e_5)$ & 0 & 0 & $-\frac{\mu^2}{4}e_5$ & 0 & $\frac{\mu^2}{4}e_3$ \\
    \hline
    $R(e_4,e_5)$ & 0 & $\frac{\lambda\mu}{4}e_3$ & $-\frac{\lambda\mu}{4}e_2$ & 0 & 0 \\
    \hline
  \end{tabular}
\end{eqnarray}

For $A$ and $B$ similar to the equations \ref{AB} we have
\begin{eqnarray}
  R(B,A)A &=&\frac{3\lambda^2}{4}(b\tilde{a}-a\tilde{b})(ae_2-be_1)+\frac{3\mu^2}{4}(c\tilde{a}-a\tilde{c})(ae_3-ce_1)\nonumber\\
  &&+\frac{\lambda^2}{4}\{(d\tilde{a}-a\tilde{d})(de_1-ae_4)+(d\tilde{b}-b\tilde{d})(de_2-be_4)\} \nonumber\\
  &&+\frac{\mu^2}{4}\{(f\tilde{a}-a\tilde{f})(fe_1-ae_5)+(f\tilde{c}-c\tilde{f})(fe_3-ce_5)\}\\
  &&+\frac{\lambda\mu}{4}\{(c\tilde{b}-b\tilde{c})(de_5-fe_4)+(f\tilde{b}-b\tilde{f})(de_3-ce_4)\nonumber\\
  &&+(d\tilde{c}-c\tilde{d})(fe_2-be_5)+(f\tilde{d}-d\tilde{f})(be_3-ce_2)\}\nonumber
\end{eqnarray}
This shows that for the plan $sapn\{A,B\}$ the sectional curvature
is of the following form.
\begin{eqnarray}
  K^R(A,B) &=& -\frac{3}{4}\{\lambda^2(b\tilde{a}-a\tilde{b})^2+\mu^2(c\tilde{a}-a\tilde{c})^2\} \nonumber\\
  &&+\frac{\lambda^2}{4}\{(d\tilde{a}-a\tilde{d})^2+(d\tilde{b}-b\tilde{d})^2\} \\
  &&+\frac{\mu^2}{4}\{(f\tilde{a}-a\tilde{f})^2+(f\tilde{c}-c\tilde{f})^2\}\nonumber\\
  &&+\frac{\lambda\mu}{2}\{(f\tilde{b}-b\tilde{f})(d\tilde{c}-c\tilde{d})+(c\tilde{b}-b\tilde{c})(d\tilde{f}-f\tilde{d})\}\nonumber
\end{eqnarray}

The scalar curvature of this Riemannian manifold is

\begin{eqnarray}\label{scalar case 2}
  S(p)=-\frac{\lambda^2+\mu^2}{2}<0.
  \end{eqnarray}

The last equation shows that simply connected two-step nilpotent
Lie groups of dimension five equipped with left-invariant
Riemannian metrics with two dimensional center have constant
negative scalar curvature.\\
Also similar to pervious case we have the following proposition.
\begin{prop}
There is not any left invariant Randers metric of Berwald type on
simply connected two-step nilpotent Lie groups of dimension five
with $2-$dimensional center.
\end{prop}

\section{\textbf{Lie algebras with 3-dimensional center}}

Now we study simply connected two-step nilpotent Lie groups of
dimension five equipped with left-invariant Riemannian metrics
with three dimensional center. This case is different from the
pervious cases because in this case the Lie group $N$ admits a
family of
invariant Randers metrics of Berwald type.\\
The Lie algebra structure of this case is as follows. \\
$\frak{n}$ admits an orthonormal basis $\{e_1,e_2,e_3,e_4,e_5\}$
such that for $\lambda>0$
\begin{eqnarray}
  [e_1,e_2]=\lambda e_3,
\end{eqnarray}
where $\{e_3,e_4,e_5\}$ is a basis for the center of $\frak{n}$
and the other commutators are zero (see \cite{HoKo}.).\\

The Levi-Civita connection of the Riemannian manifold can be
obtained by the equation \ref{Levi-Civita} as the following table.
\begin{eqnarray}\label{Levi-Civita case3}
 \begin{tabular}{|c|c|c|c|c|c|}
  \hline
    & $e_1$ & $e_2$ & $e_3$ & $e_4$ & $e_5$ \\
  \hline
  $\nabla_{e_1}$ & 0 & $\frac{\lambda}{2}e_3$ & $-\frac{\lambda}{2}e_2$ & 0 & 0 \\
  \hline
  $\nabla_{e_2}$ & $-\frac{\lambda}{2}e_3$ & 0 & $\frac{\lambda}{2}e_1$ & 0 & 0 \\
  \hline
  $\nabla_{e_3}$ & $-\frac{\lambda}{2}e_2$ & $\frac{\lambda}{2}e_1$ & 0 & 0 & 0 \\
  \hline
  $\nabla_{e_4}$ & 0 & 0 & 0 & 0 & 0 \\
  \hline
  $\nabla_{e_5}$ & 0 & 0 & 0 & 0 & 0 \\
  \hline
\end{tabular}
\end{eqnarray}
Similar to the above cases the curvature tensor can compute as
this table.
\begin{eqnarray}\label{Curvature case3}
  \begin{tabular}{|c|c|c|c|c|c|}
    \hline
     & $e_1$ & $e_2$ & $e_3$ & $e_4$ & $e_5$ \\
    \hline
    $R(e_1,e_2)$ & $\frac{3\lambda^2}{4}e_2$ & $-\frac{3\lambda^2}{4}e_1$ & 0 & 0 & 0 \\
    \hline
    $R(e_1,e_3)$ & $-\frac{\lambda^2}{4}e_3$ & 0 & $\frac{\lambda^2}{4}e_1$ & 0 & 0 \\
    \hline
    $R(e_1,e_4)$ & 0 & 0 & 0 & 0 & 0 \\
    \hline
    $R(e_1,e_5)$ & 0 & 0 & 0 & 0 & 0 \\
    \hline
    $R(e_2,e_3)$ & 0 & $-\frac{\lambda^2}{4}e_3$ & $\frac{\lambda^2}{4}e_2$ & 0 & 0  \\
    \hline
    $R(e_2,e_4)$ & 0 & 0 & 0 & 0 & 0  \\
    \hline
    $R(e_2,e_5)$ & 0 & 0 & 0 & 0 & 0  \\
    \hline
    $R(e_3,e_4)$ & 0 & 0 & 0 & 0 & 0  \\
    \hline
    $R(e_3,e_5)$ & 0 & 0 & 0 & 0 & 0  \\
    \hline
    $R(e_4,e_5)$ & 0 & 0 & 0 & 0 & 0  \\
    \hline
  \end{tabular}
\end{eqnarray}

Also for orthonormal set $\{A,B\}$ we have

\begin{eqnarray}
  R(B,A)A &=& \frac{\lambda^2}{4}\{3(b\tilde{a}-a\tilde{b})(ae_2-be_1)+(c\tilde{a}-a\tilde{c})(ce_1-ae_3) \\
  &&+(c\tilde{b}-b\tilde{c})(ce_2-be_3)\}\nonumber.
\end{eqnarray}
The last equation shows that the sectional curvature of the
Riemannian manifold $N$ for the plan $span\{A,B\}$is
\begin{eqnarray}
  K^R(A,B) =
  \frac{\lambda^2}{4}\{(c\tilde{a}-a\tilde{c})^2+(c\tilde{b}-b\tilde{c})^2-3(b\tilde{a}-a\tilde{b})^2\}.
\end{eqnarray}

Also by using the orthonormal basis $\{e_1,\cdots,e_5\}$ for any
$p\in N$ we have

\begin{eqnarray}
  S(p) = -\frac{\lambda^2}{2}<0.
\end{eqnarray}

The above equation shows that the Riemannian manifold $N$ is of
constant negative scalar curvature.\\

Now we discuss the left invariant Randers metrics of Berwald
type on this manifold.\\
Let $Q\in\frak{n}$ be a left invariant vector field on $N$ which
is parallel with respect to the Levi-Civita connection induced by
the left invariant Riemannian metric of $N$. By using table
\ref{Levi-Civita case3} and a simple computation we have
\begin{eqnarray}
  Q &=& q_1e_4+q_2e_5.
\end{eqnarray}

On the other hand the length of $Q$ must be less than $1$,
therefore we have $0<\sqrt{q_1^2+q_2^2}<1$.\\

\begin{prop}
Let $F$ be the Randers metric induced by the Riemannian metric $g$
and the left invariant vector field $Q$ as the formula
\ref{Randers}. Then the flag curvature of the flag $(P,A)$ in
$T_eN$ is given by
\begin{eqnarray}
  K(P,A)=\frac{K^R(A,B)}{(1+q_1d+q_2f)^2},
\end{eqnarray}
where $B$ is any vector in $P$ such that $\{A,B\}$ is an
orthonormal basis for $P$.
\end{prop}
\begin{proof}
The Randers metric is of Berwald type therefore the Levi-Civita
connection of $g$ and the Chern connection of $F$ coincide. Hence
$F$ and $g$ have the same curvature tensor. Now a direct
computation shows that

\begin{eqnarray}
  g_A(R(B,A)A,B) &=& <R(B,A)A,B>(1+<Q,A>)+<Q,B><R(B,A)A,Q+A>\nonumber \\
                 &=& K^R(A,B)(1+q_1d+q_2f),
\end{eqnarray}
\begin{eqnarray}
  g_A(A,A) = (1+<Q,A>)^2=(1+q_1d+q_2f)^2,
\end{eqnarray}

\begin{eqnarray}
  g_A(B,B) =
  1+<Q,B>^2+<Q,A>=1+(q_1\tilde{d}+q_2\tilde{f})^2+q_1d+q_2f,
\end{eqnarray}
and
\begin{eqnarray}
   g_A(A,B) =
   <Q,B>(1+<Q,A>)=(q_1\tilde{d}+q_2\tilde{f})(1+q_1d+q_2f).
\end{eqnarray}
Therefore by using formula \ref{flag} the proof is completed.
\end{proof}
The last proposition shows that the Finsler manifold $(N,F)$
admits negative, zero and positive flag curvature in any point.

\large{\textbf{Acknowledgment}}\\
This work was supported by the research grant from Shahrood
university of technology.

\bibliographystyle{amsplain}

\end{document}